\documentclass[a4paper]{amsart}
\usepackage{multicol,ifthen}
\usepackage{amsmath}
\usepackage{amssymb}
\usepackage{amsthm} 
\usepackage{enumerate}

\begin{document}

\newcommand{\F}{\mathcal{F}}
\newcommand{\R}{\mathbb R}
\newcommand{\T}{\mathbb T}
\newcommand{\N}{\mathbb N}
\newcommand{\Z}{\mathbb Z}
\newcommand{\C}{\mathbb C}  
\newcommand{\h}[2]{\mbox{$ \widehat{H^{#1}_{#2}}$}}
\newcommand{\hh}[3]{\mbox{$ \widehat{H^{#1}_{#2, #3}}$}} 
\newcommand{\n}[2]{\mbox{$ \| #1\| _{ #2} $}} 
\newcommand{\x}{\mbox{$X^r_{s,b}$}} 
\newcommand{\xx}{\mbox{$X_{s,b}$}}
\newcommand{\X}[3]{\mbox{$X^{#1}_{#2,#3}$}} 
\newcommand{\XX}[2]{\mbox{$X_{#1,#2}$}}
\newcommand{\q}[2]{\mbox{$ {\| #1 \|}^2_{#2} $}}
\newcommand{\e}{\varepsilon}
\newcommand{\lb}{\langle}
\newcommand{\rb}{\rangle}
\newcommand{\ls}{\lesssim}
\newcommand{\gs}{\gtrsim}
\newcommand{\pd}{\partial}
\newtheorem{lemma}{Lemma} 
\newtheorem{kor}{Corollary} 
\newtheorem{theorem}{Theorem}
\newtheorem{prop}{Proposition}

\title[Cauchy-problem for gKP-II equations]{On the Cauchy-problem for generalized 
Kadomtsev-Petviashvili-II equations}

\author[A.~Gr{\"u}nrock]{Axel~Gr{\"u}nrock}

\address{Axel~Gr{\"u}nrock: Rheinische Friedrich-Wilhelms-Universit\"at Bonn,
Mathematisches Institut, Beringstra{\ss}e 1, 53115 Bonn, Germany.}
\email{gruenroc@math.uni-bonn.de}

\thanks{The author was partially supported by the Deutsche Forschungsgemeinschaft, Sonderforschungsbereich 611.}


\begin{abstract}
The Cauchy-problem for the generalized 
Kadomtsev-Petviashvili-II equation
$$u_t + u_{xxx} + \partial_x^{-1}u_{yy}= (u^l)_x, \quad l \ge 3,$$
is shown to be locally well-posed in almost
critical anisotropic Sobolev spaces. The proof
combines local smoothing and maximal function
estimates as well as bilinear refinements of
Strichartz type inequalities via multilinear
interpolation in $X_{s,b}$-spaces.
\end{abstract}

\maketitle

Inspired by the work of Kenig and Ziesler \cite{KZa}, \cite{KZb} we consider the Cauchy problem

\begin{equation}
 \label{CP}
u(x,y,0)=u_0(x,y), \quad (x,y) \in \R^2
\end{equation}

for the generalized Kadomtsev-Petviashvili-II equation (for short: gKP-II)

\begin{equation}
 \label{gKP-II}
u_t + u_{xxx} + \partial_x^{-1}u_{yy}= (u^l)_x,
\end{equation}

where $l \ge 3$ is an integer. Concerning earlier results on related problems for this equation we
refer to the works of Saut \cite{S93}, I{\'o}rio and Nunes \cite{IN}, and Hayashi, Naumkin, and Saut \cite{HNS}.

\quad

For the Cauchy data we shall assume $u_0 \in H^{(s)}$, where for $s=(s_1,s_2,\e)$ the Sobolev type space $H^{(s)}$
is defined by its norm in the following way: Let $\xi:=(k,\eta)\in \R^2$ denote the Fourier variables corresponding
to $(x,y) \in \R^2$ and $\lb D_x\rb ^{\sigma_1} = \F ^{-1}\lb k\rb ^{\sigma_1}\F$, $\lb D_y\rb ^{\sigma_2} = \F ^{-1}\lb \eta\rb ^{\sigma_2}\F$, as well as $\lb D_x^{-1}D_y\rb ^{\sigma_3} = \F ^{-1}\lb k^{-1}\eta\rb ^{\sigma_3}\F$, where $\F$
denotes the Fourier transform and $\lb x\rb^{\sigma}= (1 + x^2)^{\frac{\sigma}{2}}$. Setting

$$\|u_0\|_{\sigma_1,\sigma_2,\sigma_3}:= \|\lb D_x\rb ^{\sigma_1}\lb D_y\rb ^{\sigma_2}\lb D_x^{-1}D_y\rb ^{\sigma_3} u_0\|_{L^2_{xy}}$$

we define

$$\|u_0\|_{H^{(s)}}:= \|u_0\|_{s_1+2s_2+ \e,0,0}+\|u_0\|_{s_1,s_2,\e}.$$

Almost the same data spaces are considered in \cite{KZa}, \cite{KZb}, the only new element here is the additional
parameter $s_2$, which will play a major role only for powers $l \ge 4$.

\quad

Using the contraction mapping principle we will prove a local well-posedness result for \eqref{CP}, \eqref{gKP-II} with
regularity assumptions on the data as weak as possible. Two types of estimates will be be involved in the proof: On the
one hand there is the combination of local smoothing effect (\cite{KZa}, see also \cite{IMS}) and maximal
function estimate (proven in \cite{KZa}, \cite{KZb}), which has been used in \cite{KZa}, a strategy that had been
developed in \cite{KPV93} in the context of generalized KdV equations. On the other hand we rely on Strichartz type 
estimates (cf. \cite{S93}) and especially on the bilinear refinement thereof taken from \cite{HT}, \cite{H}, see also \cite{IM01}, \cite{Ta}, \cite{Tz}, \cite{T&T01}. A similar bilinear estimate involving $y$-derivatives (to be proven below) will serve to deal with the $\lb D_y\rb$ containing part of the norm. To make these two elements meet, we use Bourgain's 
Fourier restriction norm method \cite{B93}, especially the following function spaces of $X_{s,b}$-type. The basic 
space $X_{0,b}$ is given as usual by the norm

$$\|u\|_{X_{0,b}} :=\|\lb\tau - \phi(\xi)\rb^b \F u\|_{L^2_{\xi \tau}},$$

where $\phi(\xi)=k^3- \frac{\eta^2}{k}$ is the phase function of the linearized KP-II equation. According to the data
spaces chosen above we shall also use

$$\|u\|_{X_{\sigma_1, \sigma_2,\sigma_3;b}} :=\|\lb D_x\rb ^{\sigma_1}\lb D_y\rb ^{\sigma_2}\lb D_x^{-1}D_y\rb ^{\sigma_3}u\|_{X_{0,b}} $$

as well as

$$\|u\|_{X_{(s),b}} :=\|u\|_{X_{s_1+2s_2+\e,0,0;b}} +\|u\|_{X_{s_1,s_2,\e;b}} .$$

Finally the time restriction norm spaces $X_{(s),b}(\delta)$ constructed in the usual manner will become our solution spaces.
Now we can state the main result of this note.

\begin{theorem}
 \label{lwp}
 Let $s_1 > \frac12$, $s_2 \ge \frac{l-3}{2(l-1)}$ and $0 < \e \le \min{(s_1,1)}$. Then for $s=(s_1,s_2,\e)$ and
 $u_0 \in H^{(s)}$ there exist $\delta = \delta(\|u_0\|_{H^{(s)}})>0$ and $b > \frac12$ such that there is a unique
 solution $u \in X_{(s),b}(\delta)$ of \eqref{CP}, \eqref{gKP-II}. This solution is persistent and the flow map
 $S: u_0 \mapsto u$, $H^{(s)}\rightarrow X_{(s),b}(\delta)$ is locally Lipschitz.
\end{theorem}

The lower bounds on $s_1$, $s_2$, and $\e$ are optimal (except for endpoints) in the sense that scaling considerations
strongly suggest the necessity of the condition $s_1 + 2s_2 + \e \ge \frac12 + \frac{l-3}{l-1}$. Moreover, for $l=3$
$C^2$-illposedness is known for $s_1 < \frac12$ or $\e < 0$, see \cite[Theorem 4.1]{KZa}. The affirmative result in \cite{KZa}
concerning the cubic gKP-II equation, that was local well-posedness for $s_1> \frac34$, $s_2=0$, and $\e > \frac12$ is
improved here by $\frac34$ derivatives. Some effort was made to keep the number of $y$-derivatives as small as 
possible\footnote{For $l\ge 5$ the use of local smoothing effect and maximal function estimate \emph{alone} yields LWP for
$s_1 > \frac{l-3}{l-1}$, $s_2 =0$, and $\e > \frac12$, which is optimal from the scaling point of view, too. This result may
be seen as essentially contained in \cite[Theorem 2.1 and Lemma 3.2]{KZa} plus \cite[proofs of Theorem 2.10 and Theorem 2.17]{KPV93}. This variant always requires $\frac12 +$ derivatives in $y$.}, but we shall not attempt to give evidence for the
necessity of the individual lower bounds on $s_1$, $s_2$, and $\e$, respectively.

\quad

By the general arguments concerning the Fourier restriction norm method introduced in \cite{B93} and further developed
in \cite{GTV97}, \cite{KPV96}, matters reduce to proving the following multilinear estimate:

\begin{theorem}
 \label{multest}
 Let $s_1 > \frac12$, $s_2 \ge \frac{l-3}{2(l-1)}$ and $0 < \e \le \min{(s_1,1)}$. Then there exists $b'> - \frac12$, such that
 for all $b > \frac12$ and all $u_1,\dots,u_l \in X_{(s),b}$ supported in $\{|t| \le 1\}$ the estimate
 $$ \|\partial_x \prod_{j=1}^l u_j\|_{X_{(s),b'}} \ls \prod_{j=1}^l\|u_j\|_{X_{(s),b}}$$
 holds true.
\end{theorem}

To prepare for the proof of Theorem \ref{multest} let us first recall those estimates for free solutions $W(t)u_0$ of the
linearized KP-II equation, which we take over from the literature. First we have the local smoothing estimate from
\cite[Lemma 3.2]{KZa}: For $0 \le \lambda \le 1$

\begin{equation}
 \label{ls}
\|D_x^{(1-\lambda)}(D_x^{-1}D_y)^{\lambda}W(t)u_0\|_{L_x^{\infty}L_{yt}^2}\ls \|u_0\|_{L^2_{xy}}
\end{equation}

and hence by the transfer principle \cite[Lemma 2.3]{GTV97} for $b > \frac12$

\begin{equation}
  \label{lsx}
 \|D_x^{(1-\lambda)}(D_x^{-1}D_y)^{\lambda}u\|_{L_x^{\infty}L_{yt}^2}\ls \|u\|_{X_{0,b}}.
\end{equation}

Interpolation with the trivial case $L^2_{xyt}=X_{0,0}$ and duality give for $0\le \theta \le 1$,
 $\frac{1}{p_{\theta}}=\frac{1-\theta}{2}$

\begin{equation}
  \label{lsxtheta}
 \|[D_x^{(1-\lambda)}(D_x^{-1}D_y)^{\lambda}]^{\theta}u\|_{L_x^{p_{\theta}}L_{yt}^2}\ls \|u\|_{X_{0,\theta b}}
\end{equation}

as well as 

\begin{equation}
  \label{lsxtheta'}
 \|[D_x^{(1-\lambda)}(D_x^{-1}D_y)^{\lambda}]^{\theta}u\|_{X_{0,-\theta b}}\ls \|u\|_{L_x^{p'_{\theta}}L_{yt}^2}.
\end{equation}

(For H\"older exponents $p$ we will always have $\frac1p + \frac1{p'}=1$.) To complement the local smoothing effect, we shall use the maximal function estimate

\begin{equation}
 \label{max}
\|W(t)u_0\|_{L_x^4L_{yT}^{\infty}}\ls \|\lb D_x\rb^{\frac34 +}\lb D_x^{-1}D_y\rb^{\frac12 +} u_0\|_{L^2_{xy}}
\end{equation}

due to Kenig and Ziesler \cite[Theorem 2.1]{KZa}, which is probably the hardest part of the whole story. The capital $T$
here indicates, that this estimate is only valid local in time, the $+$-signs at the exponents on the right denote
positive numbers, which can be made arbitrarily small at the cost of the implicit constant but independent of other
parameters. (This notation will be used repeatedly below.) The transfer principle implies for $u$ supported in $\{|t|\le 1\}$

\begin{equation}
 \label{maxx}
\|u\|_{L_x^4L_{yt}^{\infty}}\ls \|\lb D_x\rb^{\frac34 +}\lb D_x^{-1}D_y\rb^{\frac12 +} u\|_{X_{0,b}},
\end{equation}

where $b > \frac12$. The Strichartz type estimate

\begin{equation}
 \label{str}
\|W(t)u_0\|_{L_{xyt}^4}\ls \|u_0\|_{L^2_{xy}},
\end{equation}

taken from \cite[Proposition 2.3]{S93}, becomes

\begin{equation}
 \label{strx}
\|u\|_{L_{xyt}^4}\ls \|u\|_{X_{0,b}}, \quad b > \frac12.
\end{equation}

For its bilinear refinement

\begin{equation}
 \label{bilstr}
\|uv\|_{L_{xyt}^2}\ls \|D_x^{\frac12}u\|_{X_{0,b}}\|D_x^{-\frac12}v\|_{X_{0,b}}, \quad b > \frac12,
\end{equation}

and the dualized version thereof

\begin{equation}
 \label{bilstr'}
\|D_x^{\frac12}(uv)\|_{X_{0,-b}}\ls \|D_x^{\frac12}u\|_{X_{0,b}}\|v\|_{L_{xyt}^2}, \quad b > \frac12,
\end{equation}

we refer to \cite[Theorem 3.3 and Proposition 3.5]{H}. In order to estimate the $X_{s_1,s_2,\e;b'}$-norm of the nonlinearity
the following bilinear estimate involving $y$-derivatives will be useful. We introduce the bilinear pseudodifferential
operator $M(u,v)$ in terms of its Fourier transform

$$\widehat{M(u,v)}(\xi):= \int_{\xi=\xi_1+\xi_2}|k_1\eta -k \eta_1|^{\frac12}\widehat{u}(\xi_1)\widehat{v}(\xi_2)d\xi_1$$

and define the auxiliary space $\widehat{L}_x^pL^q_{yt}$ by $\|f\|_{\widehat{L}_x^pL^q_{yt}}:= \|\F_x f\|_{L_x^{p'}L^q_{yt}},$
where $\F_x$ denotes the partial Fourier transform with respect to the first space variable $x$ only. Then we have:

\begin{lemma}\label{lbily}
 \begin{equation}\label{bily}
  \|M(W(t)u_0,W(t)v_0)\|_{\widehat{L}_x^1L^2_{yt}} \ls \|D_x^{\frac12}u_0\|_{L^2_{xy}}\|D_x^{\frac12}v_0\|_{L^2_{xy}}.
 \end{equation}
\end{lemma}

\begin{proof}\footnote{It is not apparent from the proof, but there is a very simple idea behind Lemma \ref{lbily}. If we
take the partial Fourier transform $\F_x W(t)u_0 (k)= e^{itk^3}e^{i\frac{t}{k}\partial^2_{y}}\F_xu_0(k)$ with respect to the
$x$-variable only, we obtain a free solution of the linear Schr\"odinger equation with rescaled time variable $s=\frac{t}{k}$,
multiplied by a phase factor of size one. So any space time estimate for the Schr\"odinger equation should give a 
corresponding estimate for linearized KP-type equations. This idea was exploitet in \cite{GPS08}, \cite{G09} to obtain
suitable substitutes for Strichartz type estimates for semiperiodic and periodic problems. From this point of view Lemma
\ref{lbily} corresponds to the one-dimensional estimate $$\|D_y^{\frac12}(e^{it\partial_y^2}u_0 e^{-it\partial_y^2}v_0)\|_{L^2_{yt}} \ls \|u_0\|_{L^2_y}\|v_0\|_{L^2_y}.$$}
 We have
\begin{eqnarray*}
 & \widehat{M(W(t)u_0,W(t)v_0)}(\xi,\tau)\\
= & \int_{\xi=\xi_1+\xi_2}|k_1\eta -k \eta_1|^{\frac12}\delta(\tau -k_1^3-k_2^3+ \frac{\eta_1^2}{k_1}+\frac{\eta_2^2}{k_2})
\widehat{u_0}(\xi_1)\widehat{v_0}(\xi_2)d\xi_1.
\end{eqnarray*}
Because of $\frac{\eta_1^2}{k_1}+\frac{\eta_2^2}{k_2}=\frac{\eta^2}{k}+\frac{k}{k_1k_2}(\eta_1-\frac{k_1}{k}\eta)^2$ and
with $a=\tau -k_1^3-k_2^3+ \frac{\eta^2}{k}$ as well as $g(\eta_1)=\frac{k}{k_1k_2}(\eta_1-\frac{k_1}{k}\eta)^2-a$ this
equals
$$\int_{\xi=\xi_1+\xi_2}|k_1\eta -k \eta_1|^{\frac12}\delta(g(\eta_1))\widehat{u_0}(\xi_1)\widehat{v_0}(\xi_2)d\eta_1 dk_1.$$
The zero's of $g$ are $\eta_1^{\pm}=\frac{k_1}{k}\eta\pm \sqrt{\frac{k_1k_2a}{k}}$ and for the derivative we have
$|g'(\eta_1)|=\frac{2}{|k_1k_2|}|k_1\eta -k \eta_1|$. So we get the two contributions
$$I^{\pm}(\xi,\tau)=\frac12 \int_{k_1+k_2=k}\frac{|k_1k_2|}{|k_1\eta -k \eta_1^{\pm}|^{\frac12}}
\widehat{u_0}(k_1,\eta_1^{\pm})\widehat{v_0}(k_2,\eta-\eta_1^{\pm}) dk_1.$$
By Minkowski's integral inequality
$$\|I^{\pm}(\xi,\cdot)\|_{L^2_{\tau}} \ls \int_{k_1+k_2=k}|k_1k_2|\||k_1\eta -k \eta_1^{\pm}|^{-\frac12}
\widehat{u_0}(k_1,\eta_1^{\pm})\widehat{v_0}(k_2,\eta-\eta_1^{\pm})\|_{L^2_{\tau}}dk_1,$$
where, with $\lambda:=\sqrt{\frac{k_1k_2a}{k}}$, the square of the last $L^2_{\tau}$-norm equals
\begin{eqnarray*}
 & \int |k_1\eta -k \eta_1^{\pm}|^{-1}|\widehat{u_0}(k_1,\frac{k_1}{k}\eta\pm \lambda)
\widehat{v_0}(k_2,\frac{k_2}{k}\eta\mp \lambda)|^2 d\tau \\
\le & \frac{2}{|k_1k_2|}\int|\widehat{u_0}(k_1,\frac{k_1}{k}\eta\pm \lambda)
\widehat{v_0}(k_2,\frac{k_2}{k}\eta\mp \lambda)|^2 d\lambda,
\end{eqnarray*}
since $d\tau =2\frac{\lambda k}{|k_1k_2|}d \lambda = \mp\frac{2}{k_1k_2}(k_1\eta -k \eta_1^{\pm})d \lambda$. This gives
$$\|I^{\pm}(\xi,\cdot)\|_{L^2_{\tau}} \ls \int_{k_1+k_2=k}|k_1k_2|^{\frac12}
\left( \int|\widehat{u_0}(k_1,\frac{k_1}{k}\eta\pm \lambda)
\widehat{v_0}(k_2,\frac{k_2}{k}\eta\mp \lambda)|^2 d \lambda \right)^{\frac12} dk_1.$$
Using Parseval and again Minkowski's inequality for the $L^2_{\eta}$-norm we arrive at
\begin{eqnarray*}
 & \|\F_x M(W(t)u_0,W(t)v_0)(k)\|_{L^2_{yt}}\\
\ls & \int_{k_1+k_2=k}|k_1k_2|^{\frac12}\left( \int|\widehat{u_0}(k_1,\frac{k_1}{k}\eta\pm \lambda)
\widehat{v_0}(k_2,\frac{k_2}{k}\eta\mp \lambda)|^2 d \lambda d\eta \right)^{\frac12} dk_1 \\
= & \int_{k_1+k_2=k}|k_1k_2|^{\frac12} \|\widehat{u_0}(k_1,\cdot)\|_{L^2_{\eta}}\|\widehat{v_0}(k_2,\cdot)\|_{L^2_{\eta}}dk_1
\ls \|D_x^{\frac12}u_0\|_{L^2_{xy}}\|D_x^{\frac12}v_0\|_{L^2_{xy}}
\end{eqnarray*} 
by Cauchy-Schwarz and a second application of Parseval's identity. 
\end{proof}
\begin{kor} Let $b> \frac12$. Then
 \begin{equation}\label{bilyx}
  \|M(u,v)\|_{\widehat{L}_x^1L^2_{yt}} \ls \|D_x^{\frac12}u\|_{X_{0,b}}\|D_x^{\frac12}v\|_{X_{0,b}}
 \end{equation}
and
 \begin{equation}\label{bilyx'}
  \|D_x^{-\frac12}M(u,v)\|_{X_{0,-b}}\ls \|u\|_{\widehat{L}_x^{\infty}L^2_{yt}}\|D_x^{\frac12}v\|_{X_{0,b}}.
 \end{equation}
\end{kor}

\begin{proof}
 Lemma \ref{lbily} implies \eqref{bilyx} via the transfer principle, \eqref{bilyx'} is then obtained by duality. In fact, if
 we fix $v$ and consider the linear map $M_v(u):=M(u,v)$, then its adjoint is given by $M^*_v=M_{\overline{v}}$, and we have
 $\|\overline{v}\|_{X_{0,b}}=\|v\|_{X_{0,b}}$.
\end{proof}

Now we are prepared for the proof of the central multilinear estimate.

\begin{proof}[Proof of Theorem \ref{multest}] \hfill \\
 \quad

1. We use \eqref{lsxtheta'} with $\lambda = 0$, $\theta = \frac12$ and H\"older to obtain for $b_0 < - \frac14$

\begin{eqnarray*}
 & \|D_x^{\frac12} \prod_{j=1}^l u_j\|_{X_{0,b_0}}\ls \|\prod_{j=1}^l u_j\|_{L_x^{\frac43}L^2_{yt}}\\
\ls & \|u_1\|_{L_x^{4}L^2_{yt}}\|u_2\|_{L_x^{4}L^{\infty}_{yt}}\|u_3\|_{L_x^{4}L^{\infty}_{yt}}\prod_{j \ge 4}\|u_j\|_{L^{\infty}_{xyt}}.
\end{eqnarray*}

For the first factor we have by \eqref{lsxtheta}, again with $\lambda = 0$, $\theta = \frac12$,

$$\|u_1\|_{L_x^{4}L^2_{yt}} \ls \|D_x^{-\frac12}u_1\|_{X_{0,\frac14+}},$$

while the second and third factor are estimated by \eqref{maxx}

$$\|u_{2,3}\|_{L_x^{4}L^{\infty}_{yt}}\ls \|\lb D_x\rb^{\frac34 +}\lb D_x^{-1}D_y\rb^{\frac12 +} u_{2,3}\|_{X_{0,b}},$$

where $b > \frac12$. It is here that the time support assumption is needed. For $j \ge 4$ we use Sobolev embeddings in all variables to obtain for $b > \frac12$

$$\|u_j\|_{L^{\infty}_{xyt}} \ls \|\lb D_x\rb^{\frac12 +}\lb D_y\rb^{\frac12 +} u_{j}\|_{X_{0,b}}.$$

Summarizing we have for $b_0 < - \frac14$, $b > \frac12$

\begin{eqnarray}
 \label{1}
\|D_x^{\frac12} \prod_{j=1}^l u_j\|_{X_{0,b_0}}\ls \|D_x^{-\frac12}u_1\|_{X_{0,b}}
\|\lb D_x\rb^{\frac34 +}\lb D_x^{-1}D_y\rb^{\frac12 +} u_{2}\|_{X_{0,b}} \\ 
\times  \|\lb D_x\rb^{\frac34 +}\lb D_x^{-1}D_y\rb^{\frac12 +} u_{3}\|_{X_{0,b}}
\prod_{j \ge 4}\|\lb D_x\rb^{\frac12 +}\lb D_y\rb^{\frac12 +} u_{j}\|_{X_{0,b}}. \nonumber
\end{eqnarray}

2. Combining the dual version \eqref{bilstr'} of the bilinear estimate with H\"older's inequality, \eqref{bilstr} and
Sobolev embeddings we obtain for $b_1 < - \frac12$, $b > \frac12$

\begin{eqnarray}
 \label{2}
&\|D_x^{\frac12} \prod_{j=1}^l u_j\|_{X_{0,b_1}}\ls \|D_x^{\frac12}u_3\|_{X_{0,b}}\|u_1u_2\|_{L^{2}_{xyt}}\prod_{j \ge 4}\|u_j\|_{L^{\infty}_{xyt}} \\
 \ls &\|D_x^{-\frac12}u_1\|_{X_{0,b}}\|D_x^{\frac12}u_2\|_{X_{0,b}}\|D_x^{\frac12}u_3\|_{X_{0,b}}
\prod_{j \ge 4}\|\lb D_x\rb^{\frac12 +}\lb D_y\rb^{\frac12 +} u_{j}\|_{X_{0,b}}. \nonumber
\end{eqnarray}

3. Bilinear interpolation involving $u_2$ and $u_3$ gives

\begin{eqnarray*}
 \|D_x^{\frac12} \prod_{j=1}^l u_j\|_{X_{0,b'}} & \ls & \|D_x^{-\frac12}u_1\|_{X_{0,b}}
 \|\lb D_x\rb^{\frac12 +\frac{\theta}{4} +}\lb D_x^{-1}D_y\rb^{\frac{\theta}{2} +} u_{2}\|_{X_{0,b}} \\
& \times & \|\lb D_x\rb^{\frac12 +\frac{\theta}{4} +}\lb D_x^{-1}D_y\rb^{\frac{\theta}{2} +} u_{3}\|_{X_{0,b}}
\prod_{j \ge 4}\|\lb D_x\rb^{\frac12 +}\lb D_y\rb^{\frac12 +} u_{j}\|_{X_{0,b}},
\end{eqnarray*}

where $0 < \theta \ll 1$ and $b' = \theta b_0 + (1-\theta)b_1$. Now symmetrization via $(l-1)$-linear interpolation
among $u_2, \dots, u_l$ yields

\begin{equation}
 \label{3}
\|D_x^{\frac12} \prod_{j=1}^l u_j\|_{X_{0,b'}}  \ls  \|D_x^{-\frac12}u_1\|_{X_{0,b}}\prod_{j \ge 2}
\|\lb D_x\rb^{\alpha_1 +}\lb D_y\rb^{\alpha_2 +}\lb D_x^{-1}D_y\rb^{\alpha_3 +} u_{j}\|_{X_{0,b}},
\end{equation}

with $b,b'$ as before, $\alpha_1 = \frac12 + \frac{\theta}{2(l-1)}$, $\alpha_2 = \frac{l-3}{2(l-1)}$, and $\alpha_3 = \frac{\theta}{l-1}$. Using $\lb \eta \rb \le \lb k \rb\lb k^{-1}\eta\rb$, we may replace $\alpha_2 +$ by $\alpha_2 $
in \eqref{3}. Now, for given $s_1 > \frac12$ and $\e > 0$ we choose $\theta$ close enough to zero, so that $\alpha_1 < s_1$
and $\alpha_3 < \e$, and $b_0$ (respectively $b_1$) close enough to $-\frac14$ (respectively to $-\frac12$), so that
$b' > - \frac12$. Then, assuming by symmetry that $u_1$ has the largest frequency with respect to the $x$-variable
(i. e. $|k_1| \ge |k_{2, \dots , l}|$), we obtain

\begin{equation}
  \label{4}
\|\partial_x \prod_{j=1}^l u_j\|_{X_{s_1+2s_2+\e,0,0;b'}}  \ls \|u_1\|_{X_{s_1+2s_2+\e,0,0;b}}
\prod_{j \ge 2}\|u_{j}\|_{X_{s_1,s_2,\e;b}}.
\end{equation}

4. The same upper bound holds for $\|\partial_x \prod_{j=1}^l u_j\|_{X_{s_1,s_2,\e;b'}}$, if $|\eta|\le |k|$ (or even if
$|\eta|\le |k|^2$), where $|k|$ (respectively $|\eta|$) are the frequencies in $x$ (respectively in $y$) of the whole product.
In the case where $|k|\le 1$ - assuming $\e \le \min{(s_1,1)}$ and $b'$ sufficiently close to $-\frac12$ - the estimate

\begin{equation}
  \label{5}
\|\partial_x \prod_{j=1}^l u_j\|_{X_{s_1,s_2,\e;b'}}\ls\prod_{j=1}^l \|u_j\|_{X_{s_1,s_2,\e;b}}
\end{equation}

is easily derived by a combination of the standard Strichartz type estimate \eqref{strx} and Sobolev embeddings. So we may
henceforth assume $|k|\ge1$, $|\eta|\ge 1$, and $|k^{-1}\eta|\ge 1$. One last simple observation concerning the estimation
of $\|\partial_x \prod_{j=1}^l u_j\|_{X_{s_1,s_2,\e;b'}}$: If we assume in addition to $|k_1| \ge |k_{2, \dots , l}|$ that
$u_1$ has a large frequency with respect to $y$, i. e. $|\eta| \ls |\eta_1|$, then from \eqref{3} we also obtain \eqref{5}.

\quad

5. It remains to estimate $\|\partial_x \prod_{j=1}^l u_j\|_{X_{s_1,s_2,\e;b'}}$ in the case where $|k_1| \ge |k_{2, \dots , l}|$ (symmetry assumption as before) and $|\eta_1|\ll|\eta|$. By symmetry among $u_{2,\dots,l}$ we may assume in addition
that $|\eta_2| \ge |\eta_{1,3,\dots,l}|$ and hence that $|k_2|\ll|k|$, because otherwise previous arguments apply with $u_1$
and $u_2$ interchanged. For this distribution of frequencies the symbol of the Fourier multiplier $M(u_1u_3 \cdots u_l,u_2)$
becomes

$$|k\eta_2-k_2\eta|^{\frac12} \sim |k\eta_2|^{\frac12} \gtrsim |k\eta|^{\frac12}.$$

Now let $P(u_1,\dots,u_l)$ denote the projection in Fourier space on $\{|\eta_2| \ge |\eta_{1,3,\dots,l}|\}\cap
\{\lb k_2 \rb \ll|k|\ls k_1\}$. Then by \eqref{bilyx'} we obtain for $b>\frac12$, $b_1 < -\frac12$

\begin{eqnarray*}
 & \|D_y^{\frac12}P(u_1,\dots,u_l)\|_{X_{0,b_1}}\ls \|D_x^{\frac12}u_2\|_{X_{0,b}}\|u_1u_3 \cdots
 u_l\|_{\widehat{L}_x^{\infty}L^2_{yt}}\\
\ls &\|D_x^{\frac12}u_2\|_{X_{0,b}}\|\lb D_x \rb^{\frac12 +}(u_1u_3)\|_{L^2_{xyt}}\prod_{j\ge 4} 
  \|u_j\|_{\widehat{L}_x^{\infty}L^{\infty}_{yt}}\\
\ls &\| D_x ^{0 +}u_1\|_{X_{0,b}}\|\lb D_x \rb^{\frac12 +}u_2\|_{X_{0,b}}\|\lb D_x \rb^{\frac12 +}u_3\|_{X_{0,b}}
\prod_{j\ge 4} \|\lb D_x \rb^{\frac12 +}\lb D_y \rb^{\frac12 +}u_j\|_{X_{0,b}},
\end{eqnarray*}

where besides Sobolev type inequalities we have used \eqref{bilstr} in the last step. Interpolation with \eqref{2} gives
for $0 \le \lambda \le 1$

\begin{eqnarray*}
 &\|D_x^{\frac{\lambda}{2}}D_y^{\frac{1-\lambda}{2}}P(u_1,\dots,u_l)\|_{X_{0,b_1}}  
  \ls  \|D_x^{-\frac{\lambda}{2}+}u_1\|_{X_{0,b}} \|\lb D_x \rb^{\frac12 +}u_2\|_{X_{0,b}}\\
\times &
 \|\lb D_x \rb^{\frac12 +}u_3\|_{X_{0,b}}
\prod_{j\ge 4} \|\lb D_x \rb^{\frac12 +}\lb D_y \rb^{\frac12 +}u_j\|_{X_{0,b}}.
\end{eqnarray*}

Yet another interpolation - now with \eqref{1} - gives for $0< \theta \ll 1$, $b'=\theta b_0 + (1-\theta)b_1$,
$s_x=\frac12 (\lambda(1-\theta)+\theta)$ and $s_y=\frac12(1-\theta)(1-\lambda)$

\begin{eqnarray*}
 &\|D_x^{s_x}D_y^{s_y}P(u_1,\dots,u_l)\|_{X_{0,b'}}  \ls  \|D_x^{-s_x+}u_1\|_{X_{0,b}}
\|\lb D_x \rb^{\frac12 +\frac{\theta}{4}+}\lb D_x^{-1}D_y \rb^{\frac{\theta}{2} +}u_2\|_{X_{0,b}}\\
\times &\|\lb D_x \rb^{\frac12 +\frac{\theta}{4}+}\lb D_x^{-1}D_y \rb^{\frac{\theta}{2} +}u_3\|_{X_{0,b}}
\prod_{j\ge 4} \|\lb D_x \rb^{\frac12 +}\lb D_y \rb^{\frac12 +}u_j\|_{X_{0,b}}.
\end{eqnarray*}

The next step is to equidistribute the $\lb D_y \rb$'s on $u_2,\dots,u_l$. Here we must be careful, because the symmetry
between $u_2$ and $u_3,\dots,u_l$ was broken. But since $u_2$ has the largest $y$-frequency we may first shift a
$\lb D_y \rb^{\frac{l-3}{2(l-1)}+}$ onto $u_2$ and then interpolate among $u_3,\dots,u_l$ in order to obtain

\begin{eqnarray}\label{6}
&\|D_x^{s_x}D_y^{s_y}P(u_1,\dots,u_l)\|_{X_{0,b'}}  \ls  \|D_x^{-s_x+}u_1\|_{X_{0,b}} \\
\times & \prod_{j\ge 2} \|\lb D_x \rb^{\beta_1 +}\lb D_y \rb^{\beta_2 +}\lb D_x^{-1}D_y \rb^{\beta_3 +}u_j\|_{X_{0,b}}, \nonumber
\end{eqnarray}

where $\beta_1 = \frac12 + \frac{\theta}{4}$, $\beta_2=\frac{l-3}{2(l-1)}$ and $\beta_3=\frac{\theta}{2}$. Again we may replace
$\beta_2 +$ by $\beta_2$. Now \eqref{6} is applied to $\lb D_x \rb^{s_1 }\lb D_y \rb^{s_2 }\lb D_x^{-1}D_y \rb^{\e} 
\partial_x P(u_1,\dots,u_l)$, where we can shift the $\lb D_x \rb^{s_1 }$ partly from the product to $u_1$ and the
$\lb D_y \rb^{s_2 }$ partly to $u_2$. Moreover, since $|k_2| \ls |k|$ and $|\eta| \ls |\eta_2|$ we have 
$|k^{-1}\eta|\ls|k_2^{-1}\eta_2|$, so that a $\lb D_x^{-1}D_y \rb^{\e - \theta}$ may be thrown from the product onto $u_2$.
The result is

\begin{eqnarray*}
 \|\partial_x P(u_1,\dots, u_l)\|_{X_{s_1,s_2,\e;b'}}\ls \|D_x^{s_1+1-\theta-2s_x+}u_1\|_{X_{0,b}}\\
\times \|\lb D_x \rb^{\beta_1 +}\lb D_y \rb^{\beta_2 + s_2 -s_y + \theta}\lb D_x^{-1}D_y \rb^{\e - \theta + \beta_3 +}u_2\|_{X_{0,b}}\\
\times  \prod_{j\ge 3}\|\lb D_x \rb^{\beta_1 +}\lb D_y \rb^{\beta_2 }\lb D_x^{-1}D_y \rb^{\beta_3 +}u_j\|_{X_{0,b}}.
\end{eqnarray*}
Here $\beta_2 \le s_2$ and by choosing $\theta < \min{(\e,\frac{2}{3(l-1)}, s_1-\frac12)}$ and $\lambda$ such that $s_y=\beta_2+\theta$
we can achieve that
\begin{itemize}
 \item $s_1+1-\theta-2s_x < s_1 + 2 s_2 + \e$,
 \item $\beta_2 + s_2 -s_y + \theta \le s_2$,
 \item $\e - \theta + \beta_3 < \e$,
\end{itemize}
as well as $\beta_1 < s_1$, $\beta_3 < \e$. This gives

$$\|\partial_x P(u_1,\dots, u_l)\|_{X_{s_1,s_2,\e;b'}}\ls
\|u_1\|_{X_{s_1+2s_2+\e,0,0;b}}\prod_{j \ge 2}\|u_{j}\|_{X_{s_1,s_2,\e;b}}$$

as desired.
\end{proof}

\end{document}